\documentclass[11pt,reqno]{amsart}
\usepackage{amssymb,dsfont,enumerate,lmodern,tikz}
\usepackage{hyperref}
\usepackage[portrait,a4paper,margin=3.5cm]{geometry}
\usetikzlibrary{arrows}
\usepackage[utf8]{inputenc}		
\usepackage[T1]{fontenc}
\usepackage[american]{babel}
\numberwithin{equation}{section}
\theoremstyle{plain}
\newtheorem{theorem}{Theorem}[section]
\newtheorem{proposition}[theorem]{Proposition}

\newtheorem{lemma}[theorem]{Lemma}
\newtheorem{conjecture}[theorem]{Conjecture} 

\newcommand{\dE}{\mathbb{E}}
\newcommand{\dN}{\mathbb{N}} 
\newcommand{\dP}{\mathbb{P}}
\newcommand{\cA}{\mathcal{A}}
\newcommand{\cB}{\mathcal{B}}
\newcommand{\cC}{\mathcal{C}}
\newcommand{\cE}{\mathcal{E}}
\newcommand{\cG}{\mathcal{G}}

\newcommand{\cN}{\mathcal{N}}
\newcommand{\cS}{\mathcal{S}}
\newcommand{\cU}{\mathcal{U}}
\newcommand{\veps}{\varepsilon}
\newcommand{\bx}{\mathbf x}
\renewcommand{\leq}{\leqslant}             
\renewcommand{\geq}{\geqslant}             
\begin{document} 

\title[On the spectral radius of a random matrix]{On the spectral radius of a random matrix: an upper bound without fourth moment} %
\date{Submitted: July 29, 2016. Revised: May 12, 2017} %
\thanks{Partial support: A*MIDEX project ANR-11-IDEX-0001-02 funded by the
  ``Investissements d'Avenir'' French Government program, managed by the
  French National Research Agency (ANR)}

\author{Ch.~Bordenave}%
\address[Charles Bordenave]{CNRS \& Université de Toulouse, France}%
\author{P.~Caputo}%
\address[Pietro Caputo]{Università Roma Tre, Italy}%
\author{D.~Chafaï}%
\address[Djalil Chafaï]{Université Paris-Dauphine, France}%
\author{K.~Tikhomirov}%
\address[Konstantin Tikhomirov]{University of Alberta, Canada}%

\begin{abstract}
  Consider a square matrix with independent and identically distributed
  entries of zero mean and unit variance. It is well known that if the entries
  have a finite fourth moment, then, in high dimension, with high probability,
  the spectral radius is close to the square root of the dimension. We
  conjecture that this holds true under the sole assumption of zero mean and
  unit variance. In other words, that there are no
  outliers in the circular law. In this work we establish the conjecture in
  the case of symmetrically distributed entries with a finite moment of order
  larger than two. The proof uses the method of moments combined with a novel
  truncation technique for cycle weights that might be of independent
  interest.
\end{abstract}

\maketitle
\thispagestyle{empty}


\section{Introduction}

Let $X_N$ denote the random $N\times N$ matrix $(X_{i,j})_{i,j=1,\dots,N}$,
where $X_{i,j}$ are independent 
copies of a given complex valued random variable $\bx$ with mean zero and unit
variance:
\begin{equation}\label{ass}
  \dE[\bx]=0\quad\text{and}\quad\dE\big[|\bx|^2\big]=1.
\end{equation}
Let $\rho(X_N)$ denote the spectral radius of $X_N$: 
\begin{equation}\label{eq:rad}
  \rho(X_N):=\max\Big\{|\lambda|:\text{$\lambda$ eigenvalue of $X_N$}\Big\}\,.%
\end{equation}
The well known circular law states that, in probability, the empirical
distribution of the eigenvalues of $N^{-1/2}X_N$ weakly converges to the
uniform law on the unit disc of the complex plane
\cite{tao-vu-cirlaw-bis,around}. In particular, it follows that with high
probability
\begin{equation}\label{circ}
  \rho(X_N)\geq (1-\delta)\sqrt N \,,
\end{equation}
for any $\delta>0$ and large enough $N$. Here and below we say that a sequence of
events holds with high probability if their probabilities converge to one. The
corresponding upper bound on $\rho(X_N)$ has been established
by Bai and Yin \cite{baiyin} under a finite fourth moment assumption: if
$\dE[|\bx|^4]<\infty$, then with high probability $\rho(X_N)\leq (1+\delta)\sqrt
N$, for any $\delta>0$ and large enough $N$;
see also Geman and Hwang \cite{Geman-Hwang} and Geman \cite{Geman} for an
independent proof under stronger assumptions. Together with \eqref{circ}, this
says that if $ \dE[|\bx|^4]<\infty$ then, in probability, $\rho(X_N)/\sqrt
N\to 1$, as $N\to\infty$. We refer to \cite{Geman,baiyin} and references
therein for related estimates and more background and applications concerning
the spectral radius of a random matrix.
Surprisingly, there seems to be little or no discussion at all in the
literature -- even in the recent works \cite{tao} and \cite{borcap} --
about the necessity of the fourth moment assumption for the
behavior
$\rho(X_N)\sim\sqrt{N}$. 
We propose the following conjecture, which is illustrated by Figure
\ref{fig1}.

\begin{conjecture}\label{conju}
  The convergence in probability
  \begin{equation}\label{conju1}
    \lim_{N\to\infty}\frac{\rho(X_N)}{\sqrt N}=1,
  \end{equation}
  holds under the sole assumptions \eqref{ass}.
\end{conjecture}

\begin{figure}[htbp]
  \begin{center}
    \includegraphics[width =.49\textwidth]{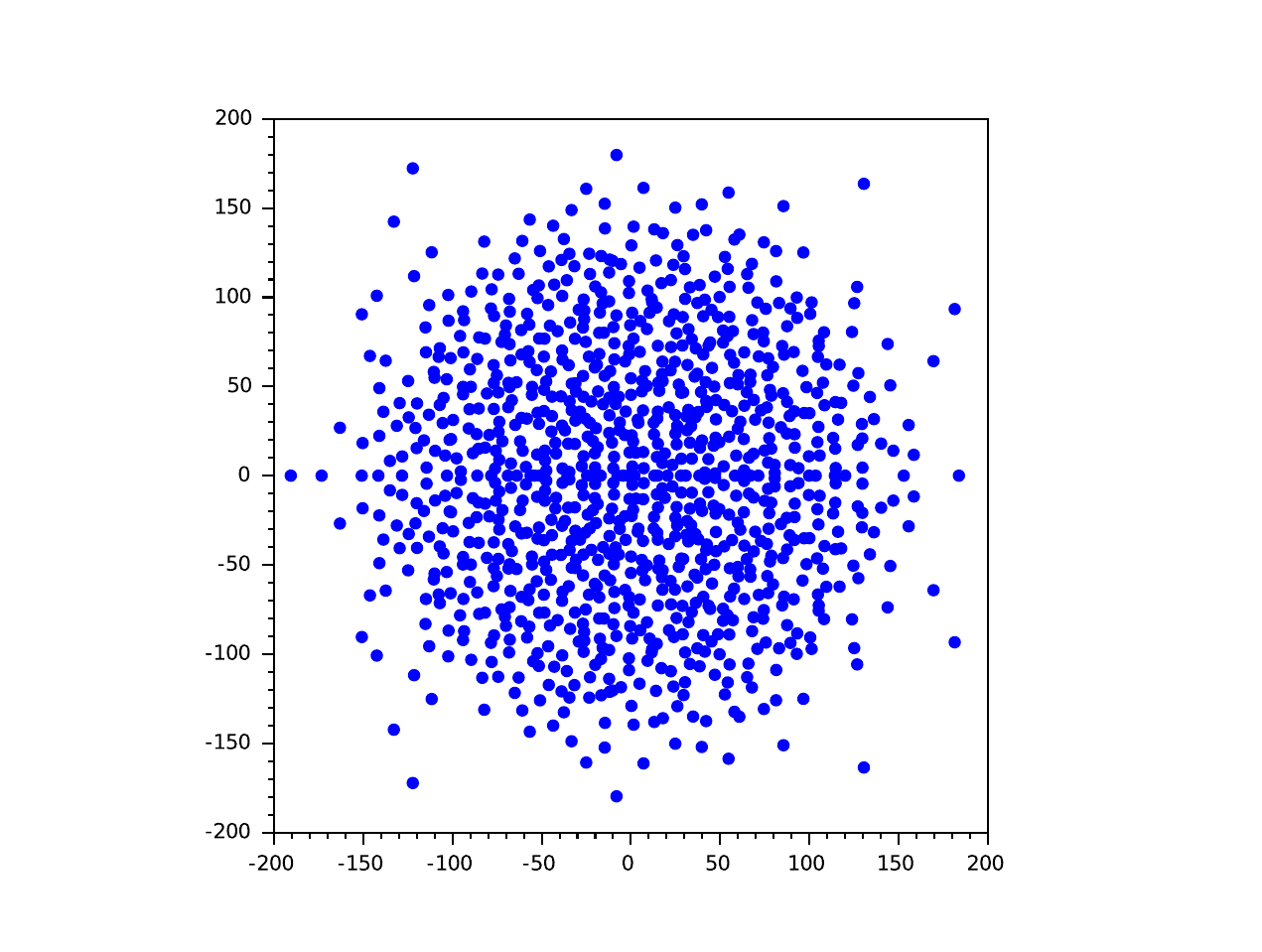}
    \includegraphics[width =.49\textwidth]{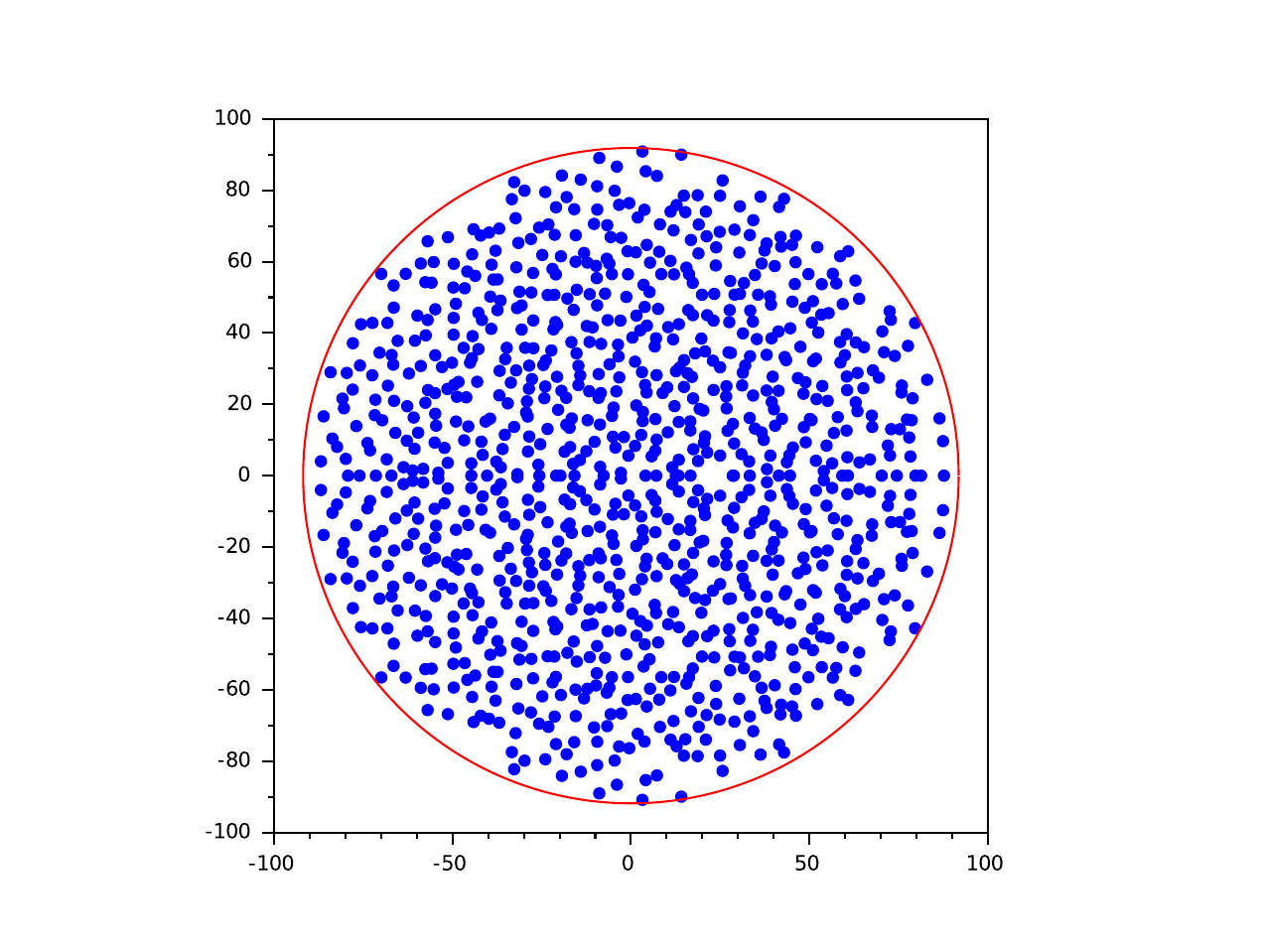}
  \end{center}
  \caption{The dots are the eigenvalues of a single realization of $X_N$ where
    $N = 1000$ and $\bx$ is real with distribution given by 
    \[
    \dP ( \bx > t ) = \dP ( \bx < - t) = \frac{1}{2t^{\alpha}},\quad t\geq1,
    \]
    with $\alpha = 1.8$ (left) and $\alpha = 2.2$ (right). The circle has
    radius $ \sqrt{ (\dE |\bx|^2 ) N}$. }
  \label{fig1}
\end{figure}

Another way to put this is to say that there are no outliers in the circular
law. This phenomenon reveals a striking
contrast between eigenvalues and singular values of $X_N$, the latter
exhibiting Poisson distributed outliers in absence of a fourth moment, see for
instance \cite{Soshnikov,Auffingeretal}. A tentative heuristic explanation of
this phenomenon may proceed as follows. Suppose $\bx$ has a heavy tail of
index $\alpha$, that is $ \dP(|\bx|>t)\sim t^{-\alpha}$, as $t\to\infty$. If
$\alpha\in(2,4)$, then with high probability in the matrix $X=X_N$ there are
elements $X_{i,j}$ with $|X_{i,j}|>N^{\beta}$, for any $1/2<\beta<2/\alpha$.
Any such element is sufficient to produce a singular value diverging as fast
as $N^\beta$. On the other hand, to create a large eigenvalue, a single large
entry is not sufficient. Roughly speaking one rather needs at least one
sequence of indices $i_1,i_2,\ldots,i_{k+1}$ with $i_1=i_{k+1}$ with a large
product $\prod_j|X_{i_j,i_{j+1}}|$, i.e.\ one cycle with a large weight if we
view the matrix as an adjacency matrix of an oriented and weighted graph. It
is not difficult to see that the sparse matrix consisting of all entries
$X_{i,j}$ with $|X_{i,j}|>N^{\beta}$ is acyclic 
with high probability, as long as $\alpha\beta>1$.

Somewhat similar phenomena should be expected for heavy tails with index
$\alpha\in(0,2)$. As shown in \cite{heavygirko}, in that case the circular law
must be replaced by a new limiting law $\mu_\alpha$ in the complex plane. More
precisely, the empirical distribution of the eigenvalues of $X/N^{1/\alpha}$
tends weakly as $N\to\infty$ to a rotationally invariant light tailed law
$\mu_\alpha$, while the empirical distribution of the singular values of
$X/N^{1/\alpha}$ tends weakly as $N\to\infty$ to a heavy tailed law
$\nu_\alpha$. By the above reasoning, no significant outliers should appear in
the spectrum. The precise analogue of \eqref{conju1} in this case is however
less obvious since the support of $\mu_\alpha$ is unbounded. From the tail of
$\mu_\alpha$, one might expect that the spectral radius is of order
$N^{1/\alpha} (\log N)^{1/\alpha +o(1)}$ while typical eigenvalues are of
order $N^{1/\alpha}$.

%

In this paper we prove that the conjectured behavior \eqref{conju1} holds if
$\bx$ is symmetric and has a finite moment of order $2+\veps$ for an arbitrary
$\veps>0$. We say that $\bx$ is symmetric if the law of $\bx$ coincides with the
law of $-\bx$.

\begin{theorem}\label{main*}
  Suppose that $\bx$ is symmetric and that $\dE\left[|\bx |^{2}\right]=1$.
  Suppose further that $\dE\left[|\bx |^{2+\veps}\right]<\infty $ for some
  $\veps>0$. Then, in probability,
  \begin{equation}\label{main1}
    \lim_{N\to\infty}\frac{\rho(X_N)}{\sqrt N}=1.
  \end{equation}
\end{theorem}

In view of \eqref{circ}, to prove the theorem one only needs to establish the
upper bound $\rho(X_N)\leq (1+\delta)\sqrt N$ with high probability, for every $\delta>0$.
We shall prove the following stronger non-asymptotic estimate, covering
variables $\bx$ whose law may depend on $N$.

\begin{theorem}\label{main}
  For any $\veps,\delta>0$ and $B>0$, there exists a constant
  $C=C(\veps,\delta,B)>0$ such that for any $N\in\dN$, for any symmetric
  complex random variable $\bx$ with $\dE\left[|\bx |^{2}\right]\leq 1$ and
  $\dE\left[|\bx |^{2+\veps}\right]\leq B$, we have
  \begin{equation}\label{main3}
    \dP\big(\rho(X_N)\geq (1+\delta)\sqrt N\big)\leq \frac{C}{(\log N)^2}.
  \end{equation}
\end{theorem}

The rest of this note is concerned with the proof of Theorem \ref{main}. We
finish this introduction with a brief overview of the main arguments involved.

\subsection{Overview of the proof}

The proof of Theorem \ref{main} combines the classical method of moments with
a novel cycle weight truncation technique. For lightness of notation, we write
$X$ instead of $X_N$. The starting point is a standard general bound on
$\rho(X)$ in terms of the trace of a product of powers of $X$ and $X^*$. Let
$\|X\|$ denote the operator norm of $X$, that is the maximal eigenvalue of
$\sqrt{X^*X}$, which is also the largest singular value of $X$. Recall the
Weyl inequality $\rho(X)\leq\|X\|$. For any integer $m\geq 1$ one has
\[
\rho(X)=\rho(X^m)^{1/m}\leq \|X^m\|^{1/m}
\quad\text{and}\quad
\|X^m\|^2\leq \mathrm{Tr}((X^*)^mX^m).
\]
It follows that for any integer $k\geq 2$, setting $m=k-1$,
\begin{equation}\label{main4}
  \rho(X)^{2k-2}\leq \mathrm{Tr}((X^*)^{k-1}X^{k-1}) = \sum_{i,j}[X^{k-1}]_{i,j}[(X^*)^{k-1}]_{j,i}.
\end{equation}
Expanding the summands in \eqref{main4} one obtains
\begin{equation}\label{main5}
  \rho(X)^{2k-2}\leq \sum_{i,j}\sum_{P_1,P_2:i\mapsto j}
  w(P_1)\bar w(P_2),
\end{equation}
where the internal sum ranges over all paths $P_1$ and $P_2$ of length $k-1$ from $i$ to $j$, the weight $w(P)$ of a path $(i_1,\dots,i_{k})$ is defined by
\begin{equation}\label{main6}
  w(P):=\prod_{\ell=1}^{k-1}X_{i_{\ell},i_{\ell+1}}\,,
\end{equation}
and $\bar w(P)$ denotes the complex conjugate of $w(P)$. So far we have not
used any specific form of the matrix entries.

As a warm up, it may be instructive to analyze the following simple special
case. Assume that $X_{i,j}$ has the distribution
\begin{equation}\label{ex1}
  \bx=
  \begin{cases}
    \pm q^{-\tfrac{1-\veps}2} 
    & \text{with probability $\frac{q}{2}$},\\ 
    0 
    & \text{with probability $1-q$},
  \end{cases}
\end{equation}
where $q=q_N\in(0,1]$ is a parameter that may depend on $N$, while
$\veps\in(0,1)$ is a fixed small constant. If $q_N\equiv 1$, then we have a
uniformly random $\pm1$ matrix, while if $q_N\to 0$, $N\to\infty$ one has a
matrix that may serve as a toy model for the sparse matrices from the
heuristic discussion given above. Notice that the assumptions of Theorem
\ref{main} are satisfied with the same parameter $\veps$ and with $B=1$, since
\[
\dE\left[|\bx |^{2}\right]=q^{\veps}%
\quad\text{and}\quad%
\dE\left[|\bx |^{2+\veps}\right]\leq q^{\veps/2}.
\]
We can now take expectation in \eqref{main5}. Using the symmetry of $\bx$ we
may restrict the sum over paths $P_1,P_2$ satisfying the constraint that in
the union $P_1\cup P_2$ each directed edge $(i_\ell,i_{\ell+1})$ appears an
even number of times. We say that $P_1\cup P_2$ is {\em even}. In this case
$\dE[w(P_1)\bar w(P_2)]= q^{-(1-\veps)(k-1)}q^{n}$, where $n$ is the number of
edges in $P_1\cup P_2$ without counting multiplicities. Let $P$ denote the
path obtained as follows: start at $i$, follow $P_1$, then add the edge
$(j,i)$, then follow $P_2$, then end with the edge $(j,i)$ again. Thus, $P$ is
an even path of length $2k$, and it is {\em closed}, 
that is, the start point and end point of $P$ coincide. Notice that
\[
\dE\left[w(P_1)\bar w(P_2)\right] \leq q^{-\veps}\dE[w(P)].
\]
Since the map $(P_1,P_2)\mapsto P$ is injective we have obtained
\begin{equation}\label{main50}
  \dE\left[\rho(X)^{2k-2}\right]\leq q^{-\veps}\sum_{P}\dE[w(P)],
\end{equation}
where the sum ranges over all even closed paths of length $2k$. Observe that
\[
\dE\left[w(P)\right]\leq q^{-(1-\veps)k}q^{\ell},
\]
where $\ell$ is the number of distinct vertices in $P$. Therefore, letting
$\cN(k,\ell)$ denote the number of even closed paths of length $2k$ with
$\ell$ vertices, \eqref{main50} is bounded above by
\begin{equation}\label{main7}
  \sum_{\ell=1}^k \cN(k,\ell)q^{-\veps}q^{-(1-\veps)k}q^{\ell}.
\end{equation}
Combinatorial estimates to be derived below, see Lemma \ref{paths and graphs}
and Lemma \ref{graphs counting}, imply that $\cN(k,\ell)\leq
k^2(4k)^{6(k-\ell)}N^\ell$. Putting all together we have found
\begin{equation}\label{main8}
  \dE\left[\rho(X)^{2k-2}\right]\leq 
  k^2 N^{k}\sum_{\ell=1}^k a(k,N,q)^{k-\ell}
\end{equation}
where $a(k,N,q)=(4k)^{6}(Nq^{(1-\veps)})^{-1}$. We choose $k\sim (\log N)^2$.
Suppose that $q\geq N^{-1-\veps}$. Then $Nq^{(1-\veps)}\geq N^{\veps^2}$ and therefore
$a(k,N,q)\leq 1$ if $N$ is large enough. It follows that
$\dE[\rho(X)^{2k-2}]\leq k^3 N^k$, and by Markov's inequality, for all fixed
$\delta>0$:
\begin{align}\label{main30}
  \dP\left(\rho(X)%
    \geq (1+\delta)\sqrt N\right)&\leq (1+\delta)^{-2k+2}N^{-k+1}\dE[\rho(X)^{2k-2}]\nonumber
  \\ & \leq (1+\delta)^{-2k+2}k^3N.
\end{align}
Since $k\sim (\log N)^2$ this vanishes faster than $N^{-\gamma}$ for any $\gamma>0$.
On the other hand, if $q\leq N^{-1-\veps}$, then a different, simpler argument
can be used. Indeed, since an acyclic matrix is nilpotent, it follows that if
$\rho(X)>0$ then there must exist a cycle with nonzero entries from the matrix
$X$. The probability of a given such cycle is $q^{\ell}$ where $\ell$ is the
number of vertices of the cycle. Estimating by $N^\ell$ the number of cycles
with $\ell$ vertices one has
\begin{equation}\label{main80}
  \dP\left[\rho(X)>0\right]\leq \sum_{\ell=1}^\infty (qN)^\ell.
\end{equation}
Thus, if $q\leq N^{-1-\veps}$, then $\dP[\rho(X)>0]\leq 2qN\leq 2N^{-\veps}$. This
concludes the proof of \eqref{main3} in the special case of the model
\eqref{ex1}.

The given argument displays, albeit in a strongly simplified form, some of the
main features of the proof of Theorem \ref{main}: the role of symmetry, the
role of combinatorics, and the fact that cycles with too high weights have to
be ruled out with a separate probabilistic estimate. The latter point requires
a much more careful handling in the general case. Since it represents the main
technical novelty of this work, let us briefly illustrate the main idea here.
Consider the collection $\cC_m$ of all possible oriented cycles with $m$ edges
of the form $C=(i_1,\dots,i_{m+1})$ with $i_j\in\{1,\dots,N\}$, and with no
repeated vertex except for $i_1=i_{m+1}$. Let $\nu_m$ denote the uniform
distribution over the set $\cC_m$. Given the matrix $X_N$, we look at the
weight $|w(C)|^{2t}$ corresponding to the cycle $C$ repeated $2t$ times, where
$w(C)$ is defined in \eqref{main6}. Since one can restrict to even closed
paths, and each such path can be decomposed into cycles that are repeated an
even number of times, it is crucial to estimate the empirical averages
\[
\nu_m\left[|w(C)|^{2t}\right] = \frac1{|\cC_m|}\sum_{C\sim\cC_m}|w(C)|^{2t},
\]
where the sum runs over all cycles with $m$
edges and $|\cC_m|$ denotes the total number of them. 
Broadly speaking, we will define an event $\cE_k$ by requiring that 
\begin{equation}\label{mainid}
  \nu_m\left[|w(C)|^{2}\right]\leq k^2\,,\quad \text{and}
  \quad\nu_m\left[|w(C)|^{2+\veps}\right]\leq k^2 B^m,
\end{equation}
for all $m\leq k$, where as before $k\sim(\log N)^2$. The assumptions of
Theorem \ref{main} ensure that $\cE_k$ has large probability by a first moment
argument. Thus, in computing the expected values of $w(P)$ we may now
condition on the event $\cE_k$. Actually, on the event $\cE_k$ we will be able
to estimate deterministically the quantities $\nu_m\left[|w(C)|^{2t}\right]$.
To see this, observe that if 
\[
w_{\max}:= \max_{C\sim\cC_m}|w(C)|
\]
denotes the maximum weight for a cycle with $m$ edges, then
\[
w_{\max}^2= \Big(\max_{C\sim\cC_m}|w(C)|^{2+\veps}\Big)^{\frac1{1+\veps/2}} \leq 
\Big(\sum_{C\sim\cC_m}|w(C)|^{2+\veps}\Big)^{\frac1{1+\veps/2}}.
\]
If $\veps$ is small enough, on the event $\cE_k$, from \eqref{mainid} one has
$w_{\max}^2\leq (|\cC_m| k^2B^m)^{1-\veps/4}$. Since $|\cC_m|\leq N^m$, a simple
iteration proves that for any $t\geq 1$:
\begin{equation}\label{mainid2}
  \nu_m\left[|w(C)|^{2t}\right]\leq (k^2N^mB^m)^{t(1-\veps/4)} \leq N^{mt(1-\veps/8)},
\end{equation}
for all $N$ large enough. The bound \eqref{mainid2} turns out to be sufficient
to handle all paths $P$ of the form of a cycle $C\sim\cC_m$ repeated $2t$
times, for all $m\leq k$. To control more general even closed paths $P$ one
needs a more careful analysis involving the estimate of larger empirical
averages corresponding to various distinct cycles at the same time. We refer
to Section \ref{statistics} below for the details. The combinatorial estimates
are worked out in Section \ref{combi}.
Finally, in Section \ref{mainproof} we complete the proof of Theorem \ref{main}.

\section{Counting paths and digraphs}\label{combi}

We first introduce the basic graph theoretic terminology and then prove some
combinatorial estimates.

\subsection{Multi digraphs and even digraphs} 

For each natural $N$, $[N]$ denotes the set $\{1,2,\dots,N\}$. A directed
graph, or simply digraph, on $[N]$, is a pair $G =(V,E)$, where $V\subset[N]$
is the set of vertices and $E\subset [N]\times[N]$ is the set of directed
edges. 
We also consider multisets $E$, where a directed edge $e\in E$ appears with
its own multiplicity $n_e\in\dN$. In this case we say that $G=(V,E)$ is a
\emph{multi digraph}.
Given a vertex $v$ of a multi digraph, the out-degree $\deg_+(v)$ is the
number of edges of the form $(v,j)\in E$, counting multiplicities. Similarly,
the in-degree $\deg_-(v)$ is the number of edges of the form $(j,v)\in E$,
counting multiplicities. Notice that each loop of the form $(v,v)$ is counted
once both in $ \deg_+(v)$ and $\deg_-(v)$. 
{\tiny
  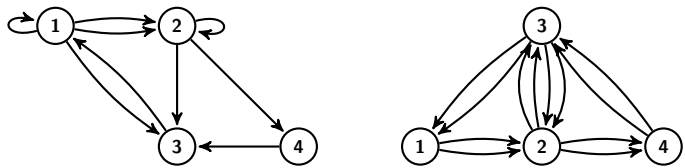
\begin{figure}[htbp]
    \begin{tikzpicture}[->,>=stealth',shorten >=.8pt,auto,node distance=1.6cm,
      thick,main node/.style={circle,draw,font=\sffamily
        \bfseries}]
      \node[main node] (1) {1};
      \node[main node] (2) [right of=1] {2};
      \node[main node] (3) [below of=2] {3};
      \node[main node] (4) [right of=3] {4};
      \node[main node] (5) [right of=4] {1};
      \node[main node] (6) [right of=5] {2};
      \node[main node] (7) [above of=6] {3};
      \node[main node] (8) [right of=6] {4};
      
      \path[every node/.style={font=\sffamily\small}]
      (1) 
      
      edge [bend right=10] node {} (2)
      edge [bend left=10] 
      node {} (2)
      edge [bend right=10] node {} (3)
      edge [loop left] node  {} (1)
      (2) 
      edge node {} (3)
      edge [loop right] node  {} (2)
      edge node {} (4)
      (3) 
      edge [bend right=10] node {} (1)
      (4) 
      edge node {} (3)
      ;

      \path[every node/.style={font=\sffamily\small}]
      (5) 
      
      edge [bend right=10] node {} (6)
      edge [bend left=10] node {} (6)
      
      (6) 
      edge [bend left=10] node {} (7)
      edge [bend left=30] node {} (7)
      
      edge [bend right=10] node {} (8)
      edge [bend left=10] node {} (8)
      (7) 
      edge [bend right=10] node {} (5)
      edge [bend left=10] node {} (5)
      
      edge [bend left=10] node {} (6)
      edge [bend left=30] node {} (6)
      (8) 
      edge [bend right=10] node {} (7)
      edge [bend left=10] node {} (7)
      
      ;
    \end{tikzpicture}
    \label{fig:fig1}
    \caption{Two examples of multi digraphs. In the first case $\deg_+(1)=4$
      and $\deg_-(1)=2$. The second example is an even digraph: it is
      generated by the even path $(1,2,3,2,4,3,1,2,3,2,4,3,1)$, and it can be
      decomposed into two double cycles, e.g.\ $(1,2,3,1,2,3,1)$ and
      $(2,4,3,2,4,3,2)$. }
  \end{figure}
}

Given natural $m$, a \emph{path} of length $m$ is a sequence
$(i_1,\dots,i_{m+1})\in[N]^{m+1}$. The path $P$ is \emph{closed} if the first
and the last vertex coincide. Each path $P=(i_1,\dots,i_{m+1})$ naturally
generates a multi digraph $G_P=(V,E)$, where $V=\{i_1,\dots,i_{m+1}\}$ and $E$
contains the edge $(i,j)$ with multiplicity $n$ if and only if the path $P$
contains exactly $n$ times the adjacent pair $(i,j)$. Notice that in general
there is more than one path generating the same multi digraph. If the path $P$
is closed, then $G_P$ is strongly connected, that is for any $u,v\in V$ one
can travel from $u$ to $v$ by following edges from $E$. A closed path without
repeated vertices except for the first and last vertices is called a {\em
  cycle}. A loop $(i,i)$ is considered a cycle of length $1$. A multi digraph
will be called a \emph{double cycle} if it is obtained by repeating two times a
given cycle. In particular, a double cycle is not allowed to have loops unless
its vertex set consists of just one vertex. We say that $P$ is an \textit{even
  path} if it is closed and every adjacent pair $(i,j)$ is repeated in $P$ an
even number of times. A multi digraph is called an \textit{even digraph} if it is
generated by an even path; see Figure \ref{fig:fig1} for an example. Thus, an
even digraph is always strongly connected. The following lemma can be proved
by adapting the classical theorems of Euler and Veblen.

\begin{lemma}\label{veblen}
  For a strongly connected multi digraph $G$, the following are equivalent:
  \begin{enumerate}[1)]
  \item $G$ is an even digraph;
  \item $\deg_+ (v) = \deg_-(v)$ is even for every vertex $v$; 
  \item $G$  can be partitioned into a collection of double cycles.
  \end{enumerate}
\end{lemma}

\subsection{Equivalence classes and rooted digraphs}

Two multi digraphs $G=(V,E)$ and $G'=(V',E')$ are called \emph{isomorphic} if
there is a bijection $f:V\to V'$ such that $(i,j)\in E$ if and only if
$(f(i),f(j))\in E'$ and the multiplicities of the corresponding edges
coincide. The associated equivalence classes are regarded as unlabeled multi
digraphs. Given an unlabeled multi digraph $\cU$, we will write $G\sim \cU$
for any multi digraph $G$ belonging to the class $\cU$. An edge-rooted multi
digraph $G = (V,E,\rho)$, or simply a \emph{rooted digraph}, is defined as a
multi digraph with a distinguished directed edge $\rho \in E$. The definition
of equivalence classes is extended to rooted digraphs as follows. Two rooted
digraphs $G=(V,E,\rho)$ and $G'=(V',E',\rho')$ are called {isomorphic} if
there is a bijection $f:V\to V'$ such that $(i,j)\in E$ if and only if
$(f(i),f(j))\in E'$, multiplicities of corresponding edges coincide, and
$f(\rho) = \rho'$. With minor abuse of notation we will use the same
terminology as above, and write $G\sim \cU$ for rooted digraphs $G$ belonging
to the equivalence class $\cU$.


\subsection{Counting}

We turn to 
the problem of estimating the number of paths generating a given even digraph, and the number of even digraphs with a given number of edges. 
Lemma \ref{paths and graphs} and Lemma \ref{graphs counting} below are combinatorial statements that appear naturally in applications of the method of moments; see, e.g., \cite{sinaisoshnikov} for somewhat related estimates.   

Let $G=(V,E)$ be an even digraph with $|E|=2k$ edges. Unless otherwise
specified, multiplicities are always included in the edge count $|E|$. By
Lemma \ref{veblen} every vertex $v$ has even in- and out-degrees satisfying
\begin{equation}\label{eqdeg}
  \deg_+(v)=\deg_-(v). 
\end{equation}
Thus $G$ has at most $k$ vertices. Moreover, since the number of edges in $G$
is $2k$, we have
\begin{equation}\label{eqdeg2}
  \sum\limits_{v\in V}\deg_+ (v)=\sum\limits_{v\in V}\deg_- (v) = 2k.
\end{equation}

\begin{lemma}[Counting paths on digraphs]\label{paths and graphs}
  Let $G=(V,E)$ be an even digraph with $|E|=2k$ and $|V|=\ell$. The number of
  paths generating $G$ does not exceed
  \[\ell (4k-4\ell )!\]
\end{lemma}

\begin{proof}
  There are $\ell$ possibilities for the starting points of the path. The path
  is then characterized by the order in which neighboring vertices are
  visited. At each vertex $v$, there are $\deg_+(v)$ visits, and at most
  $\deg_+(v) /2$ out-neighbors. If $\deg_+(v) = 2$, there is only one possible
  choice for the next neighbor. If $\deg_+ (v) \geq 4$, then there are at most
  $\deg_+(v)!$ possible choices considering all visits to the vertex $v$.
  Hence, the number of paths generating $G$ is bounded by
  \[
  \textstyle{\ell\, \prod_{v : \,\deg_+ (v) \geq 4} ( \deg_+(v) !  ) %
    \leq \ell \left( \sum_{v : \,\deg_+ (v) \geq 4} \deg_+(v) \right) !}
  \]
  where we have used that the product of factorials does not exceed the
  factorial of the sum. Now, let $q$ be the number of vertices $v$ such that
  $\deg_+ (v) \geq 4$. From \eqref{eqdeg2}, we have
  \begin{equation}\label{eq:boundqo}
    \textstyle{\sum_{v : \, \deg_+ (v) \geq 4 } \deg_+(v)     +  2   ( \ell - q ) = 2 k. }
  \end{equation}
  Estimating the sum in \eqref{eq:boundqo} from below by $4q$ one has $ 4 q +
  2 ( \ell - q) \leq 2 k. $ Hence,
  \begin{equation}\label{eq:boundq}
    q \leq k - \ell.
  \end{equation} 
  Using \eqref{eq:boundq} in \eqref{eq:boundqo} one finds
  \begin{equation*}\label{eq:summj}
    \textstyle{
      \sum_{v : \,\deg_+ (v) \geq 4 } \deg_+(v)   \leq 4  k  -4 \ell. }\end{equation*}
\end{proof}

For integers $1 \leq \ell \leq \min\{ k, N\}$, let $\cG_N (k,\ell) $ be the
set of rooted even digraphs $G=(V,E)$ with $V\subset[N]$ such that $|V|=\ell$
and $|E|=2k$.
\begin{lemma}[Graph counting]\label{graphs counting}
  For any $k,N\in\dN$, $1\leq \ell\leq \min\{k,N\}$, the cardinality of $\cG_N (k,\ell) $ satisfies 
  \begin{equation}\label{canon}
    |\cG_N(k,\ell)| \leq N^\ell k ^{2 (k - \ell)+1}. 
  \end{equation}
\end{lemma}

\begin{proof}
  We first choose $\ell$ vertices among $N$. There are
  \[
  \binom{N}{\ell}\leq\frac{N^\ell}{\ell!}
  \] 
  choices. Without loss of generality we assume that the set of vertices is
  given by $\{ 1, \ldots , \ell \}$. Next, we assign an admissible degree to
  each vertex of $\{ 1, \ldots , \ell \}$. Let $m(j)\in\dN$ be defined as
  $m(j)=\deg_\pm(j)/2$. In view of \eqref{eqdeg} and \eqref{eqdeg2}, one has
  $m(j)\geq 1$ and $\sum_{j=1}^\ell m(j)=k$. Thus there are
  \[
  \binom{k-1}{\ell-1}\leq k^{k-\ell}
  \] 
  choices for the vector $(m(1),\dots,m(\ell))$. Next, we need to count the
  number of multi digraphs with the given degree sequence. To this end, we may
  use the configuration model. Namely, we think of every vertex $j$ as having
  $m(j)$ heads and $m(j)$ tails. Altogether, there will be $k$ heads and $k$
  tails. Each head is thought of as a pair of loose out-edges (without an
  assigned out-neighbor) while each tail is thought of as a pair of loose
  in-edges (without an assigned in-neighbor). The number of multi digraphs
  with the given degree sequence is bounded by the number of bipartite
  matchings of heads and tails, which gives $k!$ possible choices. Thus, using
  $k!/\ell!\leq k^{k-\ell}$, we see that the total number of even multi
  digraphs with $\ell$ vertices and $2k$ edges is bounded above by
  \[
  N^\ell k^{2(k-\ell)}. 
  \]
  It remains to choose the root edge. Since there are at most $k$ choices, the
  proof is complete.
\end{proof}

\section{Statistics of even digraphs}

Every edge $(i,j)\in[N]\times[N]$ is given the random weight $X_{i,j}$, where
$X_{i,j}$ are independent copies of a random variable $\bx$ satisfying the
assumptions of Theorem \ref{main}. The weight of an even digraph $G=(V,E)$, is
defined as
\begin{gather}\label{pg}
  p(G) :=  \prod_{(i,j) \in E} |X_{i,j} |^{ n_{i,j} }, 
\end{gather}
where each edge $(i,j)\in E$ has multiplicity $n_{i,j}\geq 2$. Note that in
this formula we interpret ``$(i,j)\in E$'' without taking into account the
multiplicity in the multiset $E$. Given an unlabeled even graph $\cU$,
consider the equivalence class of even digraphs $\{G:\,G\sim \cU\}$. We are
interested in estimating
\begin{gather}\label{statistico}
  \cS_h(\cU):=\frac{2^h|\{G\sim \cU:\, p(G)\geq 2^h\}|}{|\{G:\,G\sim \cU\}|},
\end{gather}
for $h=0,1,2,\dots$ Moreover, we define
\begin{gather}\label{statistics}
  \cS(\cU):=\max(1,\max_{h\in\{0,1,2,\dots\}}\cS_h(\cU))\,.
\end{gather}
We refer to $S(\cU)$ as the \textit{statistics} of the unlabeled even digraph $\cU$.

We extend the above definitions to rooted even digraphs as follows. The weight
of a rooted even digraph $G = (V,E,\rho)$ is defined by
\begin{gather}\label{prg}
  p_r(G)=  \prod_{(i,j) \in E} |X_{i,j} |^{ n_{i,j} -2\mathbf{1}_{(i,j)=\rho}}.
\end{gather}
Note that 
\[
p_r(V,E,\rho) = |X_\rho|^{-2}p(V,E),
\] 
is well defined even if $X_\rho=0$ since the root edge $\rho$ satisfies $\rho\in E$
and thus $n_\rho\geq 2$. If $\cU$ is an unlabeled rooted even digraph, that is
an equivalence class of rooted even digraphs, then $\cS_h(\cU)$ and $\cS(\cU)$
are defined as in \eqref{statistico} and \eqref{statistics}, provided $p(G)$
is replaced by $p_r(G)$ in that expression. 

The special notion \eqref{prg} of rooted graph weights will be needed to handle the weight of closed paths $P$ that are obtained by artificially adding a distinguished edge; see \eqref{prwp} below. 

Estimates for the statistics $\cS(\cU)$ will be derived from a basic estimate
for double cycles. Let $\cC_m$ be the unlabeled double cycle with $2m$ edges.
Similarly, $\cC^\star_m$ will denote the unlabeled rooted double cycle with
$2m$ edges. From the assumptions of Theorem \ref{main}, for any double cycle
$C\sim \cC_m$ 
we have 
\begin{equation}\label{eqb1} 
  \mathbb{E}[ p(C)]\leq1\,,\quad \mathbb{E}[p(C)^{1+\varepsilon/2}]\leq B^m. 
\end{equation} 
Note that the same bounds apply for any rooted double cycle $C \sim
\cC_m^\star$, with the weights $p(C)$ replaced by $p_r(C)$.

\begin{lemma}[Cycle statistics]\label{cycle stats}
  For any $k\geq 1$,  
  define the event
  \[
  \cA_{k}:= \cA^1_{k}\cap\cA^2_{k}\cap\cA^3_{k}
  \]
  where
  \begin{align*}
    \cA^1_{k}&
    :=\bigcap_{m=1}^k\bigg\{\sum\limits_{h=0}^\infty\cS_h(\cC_m)\leq k^2\bigg\},\\
    \cA^2_{k}&
    :=\bigcap_{m=1}^k\bigg\{ \sum\limits_{h=0}^\infty\cS_h(\cC^\star_m)\leq k^2\bigg\},\\
    \cA^3_{k}&
    :=\bigcap_{m=1}^k\bigg\{\sum\limits_{h=0}^\infty 2^{h\varepsilon/2}\cS_h(\cC_m)\leq k^2 B^m\bigg\}.
  \end{align*}
  Then 
  \[
  \dP(\cA_{k})\geq 1-\frac{6}{k}.
  \]
\end{lemma}

\begin{proof}
  For any $a\geq 0$ one has
  \begin{equation}\label{asz}
    \frac12\sum\limits_{h=0}^\infty 2^h\mathbf{1}_{a\geq 2^h} %
    \leq a \leq 1 + 2\sum\limits_{h=0}^\infty 2^h\mathbf{1}_{a\geq 2^h}.
  \end{equation}
  Take any $C\sim\cC_m$. The first inequality in \eqref{asz} yields
  \begin{equation}\label{cycles0}
    \frac12\sum\limits_{h=0}^\infty 2^h\mathbf{1}_{p(C)\geq 2^h}\leq p(C).
  \end{equation}
  Taking the expectation,  \eqref{eqb1} implies
  \[
  \sum\limits_{h=0}^\infty 2^h\dP(p(C)\geq 2^h)\leq 2.
  \]
  On the other hand, by symmetry any $C\sim\cC_m$ satisfies
  \begin{equation}\label{cycles10}
    2^h\dP(p(C)\geq 2^h) = \dE[\cS_h(\cC_m)].
  \end{equation}
  Hence, from Markov's inequality and a union bound over $1 \leq m\leq k$, one has
  \begin{equation}\label{cycles1}
    \dP(\cA^1_{k})\geq 1-\frac{2}{k}. 
  \end{equation}
  for all $m\leq k$. 
  Next, as in \eqref{cycles0} one shows that
  \begin{equation*}\label{cycles00}
    p(C)^{1+\veps/2}\geq 
    \frac12\sum\limits_{h=0}^\infty 2^{h(1+\veps/2)}\mathbf{1}_{p(C)\geq 2^h}.
  \end{equation*}
  Then \eqref{eqb1} and \eqref{cycles10} imply
  \[
  \sum\limits_{h=0}^\infty2^{h\varepsilon/2}  \dE[\cS_h(\cC_m)]%
  =\sum\limits_{h=0}^\infty2^{h+h\varepsilon/2}\dP(p(C)\geq 2^h)%
  \leq 2\, \dE \left[ p(C)^{1+\veps/2}\right]%
  \leq 2B^m.
  \]
  Therefore, from Markov's inequality and a union bound over $1 \leq m\leq k$,
  \begin{equation}\label{cycles11}
    \dP(\cA^3_{k})\geq 1-\frac{2}{k}. 
  \end{equation}
  Finally, we observe that the same argument leading to \eqref{cycles1} can be
  repeated for rooted cycles, with no modifications. It follows that
  \begin{equation}\label{cycles011}
    \dP(\cA^2_{k})\geq 1-\frac{2}{k}. 
  \end{equation}
  From \eqref{cycles1}-\eqref{cycles011} and the union bound over $i=1,2,3$,
  it follows that 
  \[
  \dP(\cA_{k})\geq 1-\frac{6}{k}.
  \]
\end{proof}
To make the link with the arguments presented in the introduction, we remark that if $\nu_m$ denotes the uniform distribution over the set of all $C\sim \cC_m$, then \eqref{asz} allows one to interpret the events $\cA^1_k$ and $\cA^3_k$ as the condition discussed in \eqref{mainid}.  

In the remainder of this section, on the event $\cA_k$, we will
deterministically upper bound the statistics of any unlabeled rooted even
digraph; see Proposition \ref{exp prop} below. The proof will use the
following induction statement.

\begin{lemma}[Induction]\label{induction}
  Fix integers $1\leq r\leq m\leq k\ll \sqrt N$. Let $\cU'$ be an unlabeled
  rooted even digraph with at most $k$ vertices and assume that $\;\cU'$ can
  be decomposed as $\;\cU'=\cU\cup \cC_m$ for some unlabeled rooted even
  digraph $\cU$ and a double cycle $\cC_m$ of length $2m$ having $r$ common
  vertices with $\cU$. Suppose that $\cA_k$ holds. Then
  \begin{enumerate}[1)]
  \item $\mathcal{S}(\cU')\leq 3ek^2 N^r \mathcal{S}(\cU)$;
  \item If $m\log B\leq\frac{\varepsilon}{4}r\log N$, then
    $\mathcal{S}(\cU')\leq 5e k^2N^{r(1-\varepsilon/8)} \mathcal{S}(\cU)$.
  \end{enumerate}
\end{lemma}

\begin{proof}
  Fix an even rooted digraph $G'\sim \cU'$ and denote by $C\sim \cC_m$ and
  $G\sim \cU$, respectively, the double cycle with $2m$ edges and the even
  rooted digraph isomorphic to $\cU$ so that $G'=G\cup C$. Further, let $\pi$
  be a uniform random permutation of $[N]$, which we assume to be defined on a
  different probability space. Any permutation induces a mapping on rooted
  digraphs via vertex relabeling, so that the rooted digraph $\pi[G']$ is
  uniformly distributed on the set $\{H:\,H\sim \cU'\}$. Hence we may write
  \begin{gather}\label{id1}
    \cS_h(\cU')=2^h\dP_\pi(p_r(\pi[G'])\geq 2^h),\;\;h=0,1,\dots
  \end{gather}
  where $\dP_\pi$ denotes the probability w.r.t.\ the random permutation
  $\pi$. For any $a,b\geq 0$,
  \begin{align*}
    \mathbf{1}_{ab\geq 2^h}
    & = \mathbf{1}_{ab\geq 2^h}\bigg(\sum_{\ell=1}^h\mathbf{1}_{2^{\ell-1}\leq a<2^\ell} %
    +\mathbf{1}_{a<1} +\mathbf{1}_{a\geq 2^h}\bigg)\\ 
    & \leq\sum_{\ell=1}^h\mathbf{1}_{b\geq 2^{h-\ell}; \, a\geq 2^{\ell-1}} %
    + \mathbf{1}_{b\geq 2^h} +  \mathbf{1}_{a\geq 2^h}.
  \end{align*}
  Using this and $p_r(\pi[G'])=p_r(\pi[G])\,p(\pi[C])$, one may estimate
  \begin{align}\label{estima1}
    \dP_\pi\left(p_r(\pi[G'])\geq 2^h\right)
    &\leq\sum\limits_{\ell=1}^h\dP_\pi\left(p_r(\pi[G])\geq2^{h-\ell};\,p(\pi[C])\geq2^{\ell-1}\right)\nonumber\\ 
    &\qquad +\,\dP_\pi\left(p(\pi[C])\geq 2^{h}\right)+\dP_\pi\left(p_r(\pi[G])\geq 2^{h}\right).
  \end{align}
  Let us condition on a fixed realization $R$ of $\pi$ restricted to the
  vertices $V$ of $G$. Thus, $\dP_\pi(\cdot\,|\,R)$ represents a uniform
  average over all permutations that agree with the given $R$ on $V$. We write
  $C'\sim(C;R)$ for any digraph $C'$ that has the form $C'=\pi[C]$ for some
  $\pi$ that agrees with $R$ on $V$. Since $C$ has $m-r$ free vertices (those
  which do not fall into $V$), and we can pick them among $N-|V|$ available
  vertices, the cardinality of $\{C'\sim(C;R)\}$ is at least
  \[
  (N-|V|)(N-|V|-1)\cdots(N-|V|-(m-r-1)) \geq (N-k)^{(m-r)},
  \] 
  where we use that the total number of vertices satisfies $|V|+(m-r) \leq k$.
  %
  Since the number of double cycles of length $2m$ is $\binom{N}{m}(m-1)!\leq N^m$, we can 
  write for any $\tau>0$:
  \begin{align*}\dP_\pi(
    p(\pi[C])\geq \tau\,|\,R)& =
    \frac{|\{C'\sim(C;R):\,p(C')\geq \tau\}|}{|\{C'\sim(C;R)\}|}\\
    & \leq 
    (N-k)^{r-m}|\{C'\sim\cC_m:\,p(C')\geq \tau\}| \\&\leq 
    (N-k)^{r-m}N^m\dP_\pi(p(\pi[C])\geq \tau)
    \leq eN^r\dP_\pi(p(\pi[C])\geq \tau),\end{align*}
  where we use $ r\leq m\leq k\ll \sqrt N$ to bound $(1-\tfrac{k}N)^{r-m}\leq e$. 
  Since the above estimate is uniform over the realization $R$, for any $\ell=1,2,\dots,h$ we have
  \begin{align*}
    &\dP_\pi\left(p_r(\pi[G])\geq 2^{h-\ell};\;p(\pi[C])\geq 2^{\ell-1}\right)\\
    &\qquad \leq \dP_\pi\left(%
      p_r(\pi[G])\geq 2^{h-\ell}\right)\,\sup\limits_R\,\dP_\pi\left(\pi[C]\geq 2^{\ell -1}\,|\,R\right)
    \\
    &\qquad 
    \leq eN^r\dP_\pi\left(p_r(\pi[G])\geq 2^{h-\ell}\right)\dP_\pi\left(p(\pi[C])\geq 2^{\ell-1}\right).
  \end{align*}
  Using the definition of $\cS(\cU)$ and the identity \eqref{id1} applied to
  $G$ and $C$ we obtain, for all $\ell=1,\dots,h$:
  \begin{align}\label{esta1}
    \dP_\pi\left(p_r(\pi[G])\geq 2^{h-\ell};\;p(\pi[C])\geq 2^{\ell-1}\right)
    \leq eN^r2^{1-h}\cS(\cU)\cS_{\ell-1}(\cC_m).
  \end{align}
  From \eqref{estima1} one has
  \begin{align*}
    \dP_\pi(p(\pi[G'])\geq 2^h)&\leq
    eN^r2^{1-h}\cS(\cU)\sum\limits_{\ell=0}^{h-1}\cS_{\ell}(\cC_m)+
    2^{-h}\cS_h(\cC_m) + 2^{-h}\cS(\cU).
  \end{align*}
  Since $\cS(\cU)\geq 1$, on the event $\cA_k$ of Lemma \ref{cycle stats} one can estimate
  \[
  2^h\dP_\pi(p(\pi[G'])\geq 2^h)\leq
  2eN^r\cS(\cU)\sum\limits_{\ell=0}^{\infty} \cS_{\ell}(\cC_m) +\cS(\cU) \leq
  3ek^2N^r\cS(\cU).
  \] 
  Taking the supremum over $h$, the above relation proves the first assertion of the lemma.
  
  Let us prove the second assertion. On the event $\cA_k$ of Lemma \ref{cycle stats}, for any $T\in\dN$,
  \[
  \sum_{\ell=T}^\infty\cS_{\ell}(\cC_m)\leq 2^{-\veps T/2} k^2B^m.
  \]
  Fix $T= \lceil\log_2(N^{r(1-\varepsilon/8)})\rceil$. If $m\log B\leq\frac{\varepsilon}{4}r\log N$, then
  \[
  \sum_{\ell=T}^\infty\cS_{\ell}(\cC_m) \leq k^2N^{-\varepsilon r/8}.
  \]
  Estimating as in \eqref{esta1} for all $\ell\geq T+1$, we obtain
  \[
  \sum_{\ell=T+1}^h\dP_\pi\left(p_r(\pi[G])\geq 2^{h-\ell};\;p(\pi[C])\geq 2^{\ell-1}\right)%
  \leq 2^{-h+1}ek^2\cS(\cU)N^{r(1-\varepsilon/8)}.
  \]
  On the other hand, using $\dP_\pi\left(p_r(\pi[G])\geq 2^{h-\ell}\right)\leq 2^{-h+\ell}\cS(\cU)$, we find
  \[
  \sum_{\ell=1}^T\dP_\pi(p_r(\pi[G])\geq 2^{h-\ell};\;p(\pi[C])\geq 2^{\ell-1})%
  \leq 2^{-h}\cS(\cU)2^{T+1}%
  \leq 2^{-h+2}\cS(\cU)N^{r(1-\varepsilon/8)}.
  \]
  From \eqref{estima1} it follows that
  \begin{align*}
    \dP_\pi(p(\pi[G'])\geq 2^h)%
    &\leq 2^{-h+2}ek^2\cS(\cU)N^{r(1-\varepsilon/8)} +2^{-h}\cS_h(\cC_m)+2^{-h}\cS(\cU).
  \end{align*}
  On the event $\cA_k$ one has $\cS_h(\cC_m)\leq k^2\leq k^2\cS(\cU)$, and therefore
  \begin{align*}
    2^h\dP_\pi(p(\pi[G'])\geq 2^h)&\leq 
    5ek^2 N^{r(1-\varepsilon/8)} \mathcal{S}(\cU).
  \end{align*}
  Taking the supremum over $h$, we obtain the second assertion of the lemma.
\end{proof}

We turn to the main statement of this section

\begin{proposition}[Main estimate]\label{exp prop}
  Suppose $N^{\varepsilon/16}\geq5ek^2$, and let $\cU$ be an unlabeled rooted
  even graph with $2k$ edges and $x$ vertices. Define
  \[
  y_x:=\max\left(0,k-x-\frac{4k\log B}{\varepsilon\log N}\right).
  \]
  Then, on the event $\cA_k$ we have
  \[
  \mathcal{S}(\cU)\leq N^{k-x}N^{-\varepsilon y_x/16}k^2\bigl(3ek^2\bigr)^{\frac{4k\log B}{\varepsilon\log N}}.
  \]
\end{proposition}

\begin{proof}
  By Lemma \ref{veblen} we may represent $\cU$ as the union of double cycles
  $C_1,\dots,C_q$, such that:
  \begin{enumerate}[1)]
  \item $C_1$ is rooted;
  \item for all $i\in[q]$, $C_i$ has $2m_i$ edges;
  \item for $i\geq 2$, $C_i$ has $r_i\geq 1$ common vertices with  $\cup_{j=1}^{i-1}C_j$. 
  \end{enumerate}
  Define the rooted even digraphs $U_i=\bigcup_{j=1}^i C_j$, $i=1,2,\dots,q$.
  Let $\cU_i$ denote the associated equivalence classes. Let $J$ be the set of
  indices $i\geq 2$ such that 
  \[
  m_i\log B\leq\frac{\varepsilon}{4}r_i\log N.
  \]
  Since $m_i>\frac{\varepsilon\log N}{4\log B}$ for any $i\geq 2$, $i\notin
  J$, using $\sum_{i\geq 2} m_i\leq k$ we see that
  \begin{equation*}\label{eqrq} 
    \left|\{2,\dots,q\}\setminus J\right| %
    \leq
    \frac{4k\log B}{\varepsilon\log N}.
  \end{equation*} 
  Since $\cU_1$ is a rooted double cycle with at most $2k$ edges, and we are
  assuming the validity of the event $\cA_k$, by Lemma~\ref{cycle stats} we
  have $\mathcal{S}(\cU_1)\leq k^2$. Moreover, by Lemma~\ref{induction}, one
  has
  \begin{gather*}
    \cS(\cU_i)\leq 3ek^2\mathcal{S}(\cU_{i-1})N^{r_i},\qquad i\in\{2,\dots,q\}\setminus J\\
    \cS(\cU_i)\leq 5ek^2\mathcal{S}(\cU_{i-1})N^{r_i-r_i\varepsilon/8}\leq
    \mathcal{S}(\cU_{i-1})N^{r_i-r_i\varepsilon/16},\qquad i\in
    J,
  \end{gather*} 
  where we used the assumption $5ek^2\leq N^{\varepsilon/16}$. Next, observe
  that
  \[
  \sum\limits_{i=2}^q r_i=k-x.
  \]
  Thus, combining the above estimates one has
  \[
  \cS(\cU)\leq N^{k-x}N^{-\varepsilon y'/16}k^2\bigl(3ek^2\bigr)^{\frac{4k\log B}{\varepsilon\log N}},
  \]
  where $y'=\sum\limits_{i\in J}r_i$. Note that
  \[
  \sum\limits_{i\notin J }r_i\leq\sum\limits_{i\notin J }\frac{4m_i\log B}{\varepsilon\log N} %
  \leq \frac{4k\log B}{\varepsilon\log N},
  \]
  implying that $y'\geq k-x-\frac{4k\log B}{\varepsilon\log N}$. The proof is
  complete.
\end{proof}

\section{Proof of Theorem \ref{main}}\label{mainproof}


Let $\cB$ denote the event that $|X_{ij}|\leq N^{2}$ for all $(i,j) \in
[N]\times [N]$. An 
application of Markov's inequality and the assumption $\dE[| X_{ij}|^2]\leq1$
shows that $\dP(\cB)\geq 1-1/N^2$. Thus, if we define $\cE_k:=\cA_k\cap \cB$,
where $\cA_k$ is the event from Lemma \ref{cycle stats},
then 
\begin{align}\label{eqb20}
  \dP(\cE_k)\geq 1 - N^{-2} - 6k^{-1}.
\end{align}
We are going to choose eventually $k\sim (\log N)^2$. Therefore, thanks to
\eqref{eqb20}, to prove the theorem it will be sufficient to prove the
conditional statement
\begin{equation}\label{main33}
  \dP\left(\rho(X_N)\geq (1+\delta)\sqrt N\mid \cE_k\right)\leq C(\log N)^{-2}.
\end{equation}
To prove this, we estimate the conditional moments
$\dE[\rho(X_N)^{2k-2}\mid\cE_k]$. From the expansion in \eqref{main5} one has
\begin{equation}\label{main55}
  \dE[\rho(X_N)^{2k-2}\mid\cE_k]
  \leq \sum_{i,j}\sum_{P_1,P_2:i\mapsto j}
  \dE[w(P_1)\bar w(P_2)\mid\cE_k]\,,
\end{equation}
where the internal sum ranges over all paths $P_1$ and $P_2$ of length $k-1$
from $i$ to $j$, the weight $w(P)$ of a path is defined by \eqref{main6}, and
$\bar w(P)$ denotes the complex conjugate of $w(P)$. 

Notice that since $|X_{i,j}|\leq N^2$ on the event $\cE_k$, all expected
values appearing above are well defined. By the symmetry assumption we can
replace the variables $X_{i,j}$ by
\[
X'_{i,j}=\theta_{i,j}X_{i,j}
\]
where $\theta_{i,j}\in\{-1,+1\}$ are symmetric i.i.d.\ random variables,
independent from the $\{X_{i,j}\}$. Conditioning on $\cE_k$ the entries
$X'_{ij}$ are no longer independent. However, since $\cE_k$ is measurable with
respect to the absolute values $\{|X_{ij}|\}$, the signs $\theta_{i,j}$ are
still symmetric and i.i.d.\ after conditioning on $\cE_k$. It follows that
\[
\dE\left[w(P_1)\bar w(P_2)\mid\cE_k\right]=0,
\]
whenever there is an edge with odd multiplicity in $P_1\cup P_2$. Thus, in
\eqref{main55} we may restrict to $P_1,P_2$ such that each edge in $P_1\cup
P_2$ has even multiplicity. Let $P$ denote the closed path obtained as
follows: start at $i$, follow $P_1$, then add the edge $(j,i)$, then follow
$P_2$, then end with the edge $(j,i)$ again. Thus, $P$ is an even closed path
of length $2k$. Note that according to our definition \eqref{prg}, if $G_P$ is
the rooted even digraph generated by the path $P$, with root at the edge
$(j,i)$, then
\begin{equation}\label{prwp}
|w(P_1)\bar w(P_2)| =p_r(G_P).
\end{equation}
Since the map $(P_1,P_2)\mapsto P$ is injective, \eqref{main55} and \eqref{prwp} allow us to
estimate
\begin{equation}\label{main501}
  \dE\left[\rho(X_N)^{2k-2}\mid\cE_k\right]
  \leq \sum_{P}\dE\left[p_r(G_P)\mid \cE_k\right],
\end{equation}
where the sum ranges over all even closed paths $P=(i_1,\dots,i_{2k+1})$ of
length $2k$ and $G_P$ is defined as the rooted even digraph generated by the
path $P$, with root at the edge $(i_k,i_{k+1})$. By Lemma \ref{paths and
  graphs}, the sum in \eqref{main501} can be further estimated by
\begin{equation}\label{paths to graphs}
  k \sum_{x=1}^k (4k)^{4(k-x)}\!\!\! \sum_{G\in \mathcal{G}_N(k,x)}\!\!\dE\left[p_r(G)\mid \cE_k\right],
\end{equation}
where we used $x (4k-4x)!\leq k (4k)^{4(k-x)}$, and $\mathcal{G}_N(k,x)$
denotes the set of all rooted even digraphs with $2k$ edges and $x$ vertices.
Below we estimate $\sum_{G\in \mathcal{G}_N(k,x)}p_r(G)$ deterministically on
the set $\cE_k$. Using the second inequality in \eqref{asz} one has, for any
$G\in \mathcal{G}_N(k,x)$:
\[
p_r(G)\leq 1 + 2\sum_{h=0}^{\infty}2^{h}\mathbf{1}_{p_r(G)\geq 2^h}.
\]
Since on the event $\cE_k$ all entries satisfy $|X_{i,j}|\leq N^2$, it follows
that $p_r(G)\leq N^{4k-4}$. Therefore the above sum can be truncated at
\[
H:=\lfloor4k\log_2 N\rfloor.
\]
Let $\cU$ be a given equivalence class of rooted even digraphs with $x$
vertices and $2k$ edges. Summing over all $G\sim \cU$, and recalling
\eqref{statistics},
\[
\sum_{G\sim\cU}p_r(G)\leq \Big(1+2\sum_{h=0}^H\cS_h(\cU)\Big)\left|\{G\sim\cU\}\right|\leq 
3H\cS(\cU)\left|\{G\sim\cU\}\right|.
\]
From Proposition \ref{exp prop}, on the event $\cE_k$ we can then estimate
\[
\sum_{G\sim\cU}p_r(G)\leq 
3H N^{k-x}N^{-\varepsilon y_x/16}k^2\bigl(3ek^2\bigr)^{\frac{4k\log B}{\varepsilon\log N}}\left|\{G\sim\cU\}\right|,
\]
where $y_x=\max\bigl(0,k-x-\frac{4k\log B}{\varepsilon\log N}\bigr)$. Summing
over all equivalence classes $\cU$ of rooted even digraphs with $x$ vertices
with $2k$ edges, on the event $\cE_k$ one obtains
\begin{equation}\label{thstg}
  \sum_{G\in \mathcal{G}_N(k,x)}p_r(G) %
  \leq  3H N^{k-x}N^{-\varepsilon y_x/16}k^2\bigl(3ek^2\bigr)^{\frac{4k\log B}{\varepsilon\log N}}\left|\cG_N(k,x)\right|.
\end{equation}
Going back to \eqref{paths to graphs}, using \eqref{thstg}, and Lemma
\ref{graphs counting} to estimate $\left|\cG_N(k,x)\right|$, one finds
\begin{equation}\label{maina}
  \dE\left[\rho(X_N)^{2k-2}\mid\cE_k\right]
  \leq 3Hk^4N^{k}\bigl(3ek^2\bigr)^{\frac{4k\log B} {\varepsilon\log N}}
  \sum_{x=1}^k (4k)^{6(k-x)}N^{-\varepsilon y_x/16}
\end{equation}
Fix $k\sim(\log N)^2$. 
If $x\leq k-\frac{8k\log B}{\varepsilon\log N}$, then  $y_x\geq(k-x)/2$ and therefore
\[
(4k)^{6(k-x)} 
N^{-\varepsilon y_x/16} \leq  (4k)^{6(k-x)}  N^{-\varepsilon (k-x)/32} \leq 1,
\]
provided that $N$ is sufficiently large.
It follows that %
\[
\sum_{x=1}^k (4k)^{6(k-x)}N^{-\varepsilon y_x/16}
\leq k + \tfrac{8k\log B}{\varepsilon\log N}(4k)^{\frac{48k\log B}{\varepsilon\log N}}.
\]
From \eqref{maina},  for large enough $N$ and $k\sim (\log N)^2$, one has
\begin{equation}\label{mainas}
  \dE\left[\rho(X_N)^{2k-2}\mid\cE_k\right]\leq N^k (\log N)^{C\log N},
\end{equation}
where $C=C(\veps,B)>0$ is a constant depending only on $\veps,B$.   
The proof of \eqref{main33} is concluded by using 
Markov's inequality: for any $\delta>0$, 
\begin{align*}
  \dP(\rho(X_N)\geq (1+\delta)\sqrt N\mid \cE_k)%
  &\leq (1+\delta)^{-2k+2} N^{-k+1} \dE[\rho(X_N)^{2k-2}\mid\cE_k] \\
  &\leq (1+\delta)^{-2k+2}N (\log N)^{C\log N}.
\end{align*}
Since $k\sim (\log N)^2$, for fixed $\delta>0$, the expression above is
$\mathcal{O}(N^{-\gamma})$ for any $\gamma>0$. This ends the proof of Theorem
\ref{main}.
 


\end{document}